\documentclass[11pt]{amsart}

\usepackage{amssymb,amsmath,amsfonts,amsthm,mathrsfs}
\usepackage{enumerate}
\usepackage[colorlinks=true, pdfstartview=FitV, linkcolor=blue, citecolor=blue, urlcolor=blue]{hyperref}
\usepackage{graphicx,color}
\usepackage{cite}
\usepackage{indentfirst}
\allowdisplaybreaks[2]

\begin{document}
\setlength{\parindent}{2em}

\newtheorem{thm}{Theorem}
\numberwithin{thm}{section}
\newtheorem{lem}{Lemma}
\newtheorem{cor}{Corollary}

\theoremstyle{definition}
\newtheorem{definition}{Definition}

\newtheorem{remark}{Remark}
\numberwithin{lem}{section}
\numberwithin{definition}{section}
\numberwithin{equation}{section}
\numberwithin{cor}{section}

\def\rn{{\mathbb R^{n}}}
\def\lb{\left\lbrace}  \def\rb{\right\rbrace}
\def\m{\mathcal{M}}  \def\e{\text{e}}
\def\d{\displaystyle}
\def\rn{{\mathbb R^n}}  \def\sn{{\mathbb S^{n-1}}}
\def\lv{\left\vert}  \def\rv{\right\vert}  \def\v{\vert}
\def\lV{\left\Vert}  \def\rV{\right\Vert}  \def\V{\Vert}
\def\lb{\left\lbrace}  \def\rb{\right\rbrace}
\def\k{\mathcal{K}}

\title[\bf  The limiting weak type behaviors and The LOWER BOUND  ...]{\bf The limiting weak type behaviors and The LOWER BOUND FOR a new WEAK  $L\log L$ TYPE NORM OF STRONG MAXIMAL OPERATORS}
\thanks{
	\newline\indent\hspace{1mm} {\it 2010 Mathematics Subject Classification}. Primary 42B20; Secondary 42B25.
	\newline\indent\hspace{1mm} The authors were supported partly by NSFC (Nos. 11671039, 11771358, 11871101) and NSFC-DFG (No. 11761131002).}

\date{}
\author[M. Qin]{Moyan Qin}
\address{Moyan Qin:
	School of Mathematical Sciences \\
	Beijing Normal University \\
	Laboratory of Mathematics and Complex Systems \\
	Ministry of Education \\
	Beijing 100875 \\
	People's Republic of China}
\email{myqin@mail.bnu.edu.cn}

\author[ H. Wu]{ Huoxiong Wu}
\address{ Huoxiong Wu:
	 School of Mathematical Sciences\\ Xiamen University\\ Xiamen 361005\\
	People's Republic of China}
\email{huoxwu@xmu.edu.cn}

\author[Q. Xue]{Qingying Xue$^\ast$}
\address{Qingying Xue:
		School of Mathematical Sciences \\
	Beijing Normal University \\
	Laboratory of Mathematics and Complex Systems \\
	Ministry of Education \\
	Beijing 100875 \\
	People's Republic of China}
\email{qyxue@bnu.edu.cn}
	
	\thanks {*Corresponding
author, E-mail: \texttt{qyxue@bnu.edu.cn}}

\maketitle

\begin{center}
	\begin{minipage}{13cm}
	{\small {\bf Abstract}\quad
It is well known that the weak ($1,1$) bounds doesn't hold for the strong maximal operators, but it still enjoys certain weak $L\log L$ type norm inequality. Let $\Phi_n(t)=t(1+(\log^+t)^{n-1})$ and the space $L_{\Phi_n}(\rn)$ be the set of all measurable functions on $\rn$ such that
$\|f\|_{L_{\Phi_n}(\rn)} :=\|\Phi_n(|f|)\|_{L^1(\rn)}<\infty$. In this paper, we introduce a new weak norm space $L_{\Phi_n}^{1,\infty}(\rn)$, which is more larger than $L^{1,\infty}(\rn)$ space, and establish the correspondng limiting weak type behaviors of the strong maximal operators. As a corollary, we show that $ \max\lb{2^n}{((n-1)!)^{-1}},1\rb$ is a lower bound for the best constant of the $L_{\Phi_n}\to L_{\Phi_n}^{1,\infty}$ norm of the strong maximal operators. Similar results have been extended to
 the multilinear strong maximal operators.}
\end{minipage}
\end{center}

\vspace{0.2cm}


\section{Introduction}\label{sec1}

As one of the two fundamental operators in Harmonic analysis, the Hardy-Littlewood maximal operator has played very important roles in Harmonic analysis, ergodic theory and index theory. By Lebesgue differentiation theorem, it was known that the almost everywhere convergence property of some operators is closely related to whether their associated maximal operators enjoy certain weak type inequalities. Let $B(x,r)$ be a ball in $\rn$, centered at $x$ with radius $r$. Recall that the Hardy-Littlewood maximal function
\begin{equation}\label{eq11}
	M(f)(x) = \sup_{r>0}\frac{1}{|B(x,r)|}\int_{B(x,r)}|f(y)|dy
\end{equation}
and their purpose in differentiation on $\mathbb R$ were introduced by Hardy and Littlewood  \cite{HL1930}, and later extended and developed by  Wiener \cite{W1939} on $\rn$. The famous  Hardy-Littlewood-Wierer Theorem states that  $M$ is of weak $(1,1)$  type and $L^p$ bounded for $p>1$. Similar results also hold for uncentered Hardy-Littlewood maximal operator. In particular, Grafakos and Kinnunen \cite{GK1998} investigated the weak type estimates for uncentered Hardy-Littlewood maximal operator in general measure space of dimension one.

Now, we focus our concern on the the best constants problem of Hardy-Littlewood maximal operator. The best constants problem of weak endpoints estimates for Calder\'{o}n-Zygmund type operators has always attracted lots of attentions. For example, for $n=1$, Davis \cite{D1974} obtained the best constant of weak-type $(1,1)$ norm for Hilbert transform, and Melas \cite{M2003} proved that $\V M\V_{L^1\to L^{1,\infty}}=\frac{11+\sqrt{61}}{11}$. However, for $n\ge2$, things become more subtle. The upper bound of $\|M\|_{L^1\to L^{1,\infty}}$ is determined by Stein and Stromberg \cite{SS1983}. It was shown that it is less than a constant multiply $n\log n$. Since then, only tardy progress has been made. For the lower bound, it is easy to check that
$$\lim_{\lambda\to1^-}\lambda|\lb x\in\rn:M(\chi_{B(0,1)})(x)>\lambda\rb| = \|\chi_{B(0,1)}\|_{L^1(\rn)}$$
which implies that $\|M\|_{L^1\to L^{1,\infty}}\ge1$.

In 2006, Janakiraman \cite{J2006} investigated the limiting weak type behavior of $M$. He proved that
$$\lim_{\lambda\to0^+}\lambda|\lb x\in\rn:M(f)(x)>\lambda\rb|=\| f\|_{L^1(\rn)},$$ 
which again indicates that $\|M\|_{L^1\to L^{1,\infty}}\ge1$. Therefore, this gives a new way to find the lower bound of the best constant of the maximal operator, as well as some other operators, such as singular integrals, fractional integral operators, etc. See \cite{DL2017,DL2017.,GHW,HGW2019,HW2019,HW2019.,HH2008,J2004} and the references therein.

If the supremum in (\ref{eq11}) is taken over some other kinds of non-trivial bases, such as, translation in-variant basis of rectangles in the work of C\'{o}rdoba, Fefferman \cite{CF1975}, basis formed by convex bodies in \cite{Bo}, using rectangles with a side parallel to some direction (lacunary parabolic set of directions in \cite{NSW}, Cantor set of directions in \cite{KA1}, arbitrary set of directions in \cite{AS,KA2}). The strong boundedness or the weak type estimate for these new maximal operators may fail to hold in these cases. 
	
In this paper, the object of our investigation is the maximal operator associated with the translation in-variant basis of rectangles. In \rm{1935}, Jessen, Marcinkiewicz and Zygmund \cite{JMZ1935} pointed out that the following strong maximal function is not of weak type $(1,1)$, which is quite different from the classical Hardy-Littlewood maximal operator.
$$\m_n({f)}(x) = \sup_{\substack{R \ni x \\ R\in\mathcal{R}}}\frac{1}{|R|} \int_R |f(y)|dy,$$
where $\mathcal{R}$ denotes the family of all rectangles in $\rn$ with sides parallel to the axes. One may further ask why there is such a big difference between these two operators. This is mainly because the volume of a ball only depends on its one-dimensional radius, while the volume of a rectangle is related to the lengths of $n$ sides. Therefore, $M$ is essentially an operator of one parameter and $\m_n$ is an operator of $n$ parameter. 

As a replacement of weak ($1,1$) estimate, it was shown in \cite{JMZ1935} that the strong maximal operator enjoys the $L\log L$ weak type estimate as follows:
\begin{equation}\label{eq12}
	|\lb x\in\rn:\m_n(f)(x)>\lambda\rb|\lesssim\int_\rn\Phi_n\left(\frac{|f(x)|}\lambda\right)dx,
\end{equation}
where $\Phi_n(t)=t(1+(\log^+t)^{n-1})$ and $\log^+t=\max\lb\log t,0\rb$. A geometric proof of inequality (\ref{eq12}) was given by C\'ordoba and Fefferman \cite{CF1975}. It is worth pointing out that their elegant proof relies heavily on a covering lemma they established therein. This covering lemma is very important and has been widely used in many subsequent works. We refer the readers to references \cite{B1983,BK1984/85,Ch,LP2014,M2006}.

\vspace{0.2cm}

This paper is devoted to find the lower bounds for the best constant of the weak $L\log L$ type norm of the strong maximal operators. This will be done by establishing the limiting weak type behavior of $\m_n$. Since $\m_n$ is not of weak type $(1,1)$, the space $L^{1,\infty}$ and the limit of $\lambda|\lb x\in\rn:\m_n(f)(x)>\lambda\rb|$ are not suitable for our purpose. Therefore, we need to introduce the weak norm spaces $L_{\Phi_n}^{1,\infty}(\rn)$.

\begin{definition}[\bf {New weak norm spaces $L_{\Phi_n}^{1,\infty}(\rn)$}]Let $\Phi_n(t)=t(1+(\log^+t)^{n-1})$ and the space $L_{\Phi_n}(\rn)$ be the set of all measurable functions on $\rn$ such that
	$\|f\|_{L_{\Phi_n}(\rn)} :=\|\Phi_n(|f|)\|_{L^1(\rn)}<\infty$. The new weak norm space $L_{\Phi_n}^{1,\infty}(\rn)$, which is more larger than $L^{1,\infty}(\rn)$, is defined to  be the set of all measurable functions on $\rn$ such that
	\begin{equation*}\|f\|_{L_{\Phi_n}^{1,\infty}(\rn)} := \sup_{\lambda>0}\frac\lambda{1+(\log^+\frac1\lambda)^{n-1}}|\lb x\in\rn:|f(x)|>\lambda\rb| < \infty.\end{equation*}
\end{definition}

\vspace{0.2cm}

Our main results are as follows:
\begin{thm}\label{thm1}
	If $f\in L_{\Phi_n}(\rn)$, then $\m_n(f)\in L_{\Phi_n}^{1,\infty}(\rn)$ and enjoys the limiting weak type behaviors as follows:
	\begin{enumerate}
		\item[\rm{(i)}] $\displaystyle\lim_{\lambda\to0^+}\frac\lambda{1+(\log^+\frac1\lambda)^{n-1}}|\lb x\in\rn:\m_n(f)(x)>\lambda\rb|=\frac{2^n}{(n-1)!}\|f\|_{L^1(\rn)};$
		\item[\rm{(ii)}] $\displaystyle\lim_{\lambda\to\infty}\frac\lambda{1+(\log^+\frac1\lambda)^{n-1}}|\lb x\in\rn:\m_n(f)(x)>\lambda\rb|=0$.
	\end{enumerate}
\end{thm}

\vspace{0.2cm}

Denote the centered strong maximal operator by $\m_n^c$, then we have
\begin{thm}\label{thm2}
If $f\in L_{\Phi_n}(\rn)$, then $\m_n^c(f)\in L_{\Phi_n}^{1,\infty}(\rn)$ and enjoys the limiting weak type behaviors as follows:
\begin{enumerate}
\item[\rm{(i)}] $\displaystyle\lim_{\lambda\to0^+}\frac\lambda{1+(\log^+\frac1\lambda)^{n-1}}|\lb x\in\rn:\m_n^c(f)(x)>\lambda\rb|=\frac{1}{(n-1)!}\|f\|_{L^1(\rn)};$
\item[\rm{(ii)}] $\displaystyle\lim_{\lambda\to\infty}\frac\lambda{1+(\log^+\frac1\lambda)^{n-1}}|\lb x\in\rn:\m_n^c(f)(x)>\lambda\rb|=0$.
\end{enumerate}
\end{thm}

\vspace{0.2cm}
From Theorem \ref{thm1} (i), it is easy to deduce the following corollary:

\begin{cor}\label{cor1}
The best constant of $\m_n$ and $\m_n^c$ satisfies
$$\|\m_n\|_{L_{\Phi_n}(\rn)\to L_{\Phi_n}^{1,\infty}(\rn)} \ge \max\lb\frac{2^n}{(n-1)!},1\rb;\quad \|\m_n^c\|_{L_{\Phi_n}(\rn)\to L_{\Phi_n}^{1,\infty}(\rn)} \ge1.$$
\end{cor}

\vspace{0.2cm}

The organization of this paper is as follows. The proofs of Theorem \ref{thm1} and Corollary \ref{cor1} will be presented in Section \ref{sec2}. The method of the proof of Theorem \ref{thm1} also can be applied to prove Theorem \ref{thm2}, so we leave it to the readers. In Section \ref{sec3}, a discussion on multilinear strong maximal operators will be given.


\section{Proof of Theorem \ref{thm1} and Corollary \ref{cor1}}\label{sec2}

For readability, this section will be divided into four subsections. The proof of Theorem \ref{thm1} will be given in the first three subsections, and the proof of Corollary \ref{cor1} will be demonstrated in the last one.

We begin with the following lemma, which provides a foundation for our analysis.

\begin{lem}\label{lem21}
Suppose $x=(x_1,\cdots,x_n)\in\rn$, $R_\varepsilon,r_\varepsilon,c$ be three positive numbers satisfy $c>(R_\varepsilon+r_\varepsilon)^n$. Then
$$\left|\lb x:x_1,\cdots,x_n>R_\varepsilon,\prod_{k=1}^n(x_k+r_\varepsilon)<c\rb\right| = \sum_{k=1}^nB_{n,k}c(\log c)^{n-k}+(-1)^n(R_\varepsilon+r_\varepsilon)^n,$$
where $B_{n,1}=1/(n-1)!$ and $B_{n,k}$ are real finite numbers only related to $n,k$ and $R_\varepsilon+r_\varepsilon$ for $k>2$.
\end{lem}

\begin{proof}
The proof will be done by reduction on $n$. Obviously, Lemma \ref{lem21} holds when $n=1$. Now assume that the result holds for $(n-1)$-dimensional case and we need to show that it holds for the $n$ dimensional case.  
By a fundamental calculation, we have
\begin{align*}
	& \left|\lb x:x_1,\cdots,x_n>R_\varepsilon,\prod_{k=1}^n(x_k+r_\varepsilon)<c\rb\right| \\
	&= \int_{R_\varepsilon}^{\frac c{(R_\varepsilon+r_\varepsilon)^{n-1}}-r_\varepsilon}\left|\lb(x_2,\cdots,x_n):x_2,\cdots,x_n>R_\varepsilon,\prod_{k=2}^n(x_k+r_\varepsilon)<\frac c{x_1+r_\varepsilon}\rb\right|dx_1 \\
	&= \frac1{(n-2)!}\int_{R_\varepsilon}^{\frac c{(R_\varepsilon+r_\varepsilon)^{n-1}}-r_\varepsilon}\frac c{(x_1+r_\varepsilon)}\left(\log\frac c{x_1+r_\varepsilon}\right)^{n-2}dx_1 \\
	&\quad +\sum_{k=3}^nB_{n-1,k}\int_{R_\varepsilon}^{\frac c{(R_\varepsilon+r_\varepsilon)^{n-1}}-r_\varepsilon}\frac c{(x_1+r_\varepsilon)}\left(\log\frac c{x_1+r_\varepsilon}\right)^{n-k}dx_1 \\
	&\quad +(-1)^{n-1}(R_\varepsilon+r_\varepsilon)^{n-1}\left(\frac c{(R_\varepsilon+r_\varepsilon)^{n-1}}-r_\varepsilon-R_\varepsilon\right) \\
	&= \frac{c(\log c)^{n-1}}{(n-1)!}+\sum_{k=2}^nB_{n,k}c(\log c)^{n-k}+(-1)^n(R_\varepsilon+r_\varepsilon)^n.
\end{align*}
Therefore the proof of Lemma \ref{lem21} is finished by reduction.
\end{proof}

\vspace{-0.2cm}
Now we are ready to prove Theorem \ref{thm1}. We divide it into three subsections.

\subsection{$\m_n$ is of type $(L_{\Phi_n},L_{\Phi_n}^{1,\infty})$.}
\ 
\newline
\indent Note that
\begin{equation*}
	\begin{aligned}
		\log^+\frac{|f(x)|}\lambda &\le \lb
		\begin{aligned}
			\log|f(x)|+\log\frac1\lambda, \qquad |f(x)|\ge\lambda \\
			0, \qquad\qquad\quad |f(x)|<\lambda
		\end{aligned}\right. \\
		&\le \log^+|f(x)| + \log^+\frac1\lambda,
	\end{aligned}
\end{equation*}
then we have
\begin{align*}
	\left(\log^+\frac{|f(x)|}\lambda\right)^{n-1} &\le \left(\log^+|f(x)| + \log^+\frac1\lambda\right)^{n-1} \\
	&\le 2^{n-1}\max\lb(\log^+|f(x)|)^{n-1},\left(\log^+\frac1\lambda\right)^{n-1}\rb \\
	&\le 2^{n-1}\left(1+(\log^+|f(x)|)^{n-1}\right)\left(1+\left(\log^+\frac1\lambda\right)^{n-1}\right).
\end{align*}

Therefore it follows from (\ref{eq12}) that
\begin{align*}
	& |\lb x\in\rn:\m_n(f)(x)>\lambda\rb| \\
	&\le C_n'\int_\rn\frac{|f(x)|}\lambda\left(1+2^{n-1}\left(1+(\log^+|f(x)|)^{n-1}\right)\left(1+\left(\log^+\frac1\lambda\right)^{n-1}\right)\right)dx \\
	&\le 2^nC_n'\frac{1+(\log^+\frac1\lambda)^{n-1}}\lambda\|f\|_{L_{\Phi_n}(\rn)},
\end{align*}
which implies that
\begin{equation}\label{eq21}
	\frac\lambda{1+(\log^+\frac1\lambda)^{n-1}}|\lb x\in\rn:\m_n(f)(x)>\lambda\rb| \le C_n\|f\|_{L_{\Phi_n}(\rn)}
\end{equation}
for all $\lambda>0$. This completes the proof that $\m_n(f)\in L_{\Phi_n}^{1,\infty}(\rn)$ if $f\in L_{\Phi_n}(\rn)$.

\subsection{Proof of Theorem \ref{thm1} (i)}
\ 
\newline
\indent We may assume $\|f\|_{L^1(\rn)}>0$, otherwise there is nothing need to be proved.

Note that for all $0<\varepsilon\ll\max\lb\|f\|_{L^1(\rn)},1\rb$, there exists a positive real number $r_\varepsilon>1$, such that
$$\|f\|_{L_{\Phi_n}(\rn\backslash[-r_\varepsilon,r_\varepsilon]^n)}<\varepsilon.$$

Since $C([-r_\varepsilon,r_\varepsilon]^n)$ is dense in $L_{\Phi_n}([-r_\varepsilon,r_\varepsilon]^n)$, then there exists a continuous function $\widetilde f_1$ defined on $[-r_\varepsilon,r_\varepsilon]^n$ satisfying
$$\|f-\widetilde f_1\|_{L_{\Phi_n}([-r_\varepsilon,r_\varepsilon]^n)}<\varepsilon.$$
Now we denote
\begin{align*}
	f_1 &= |\widetilde f_1|+\frac\varepsilon{(2r_\varepsilon)^n}\chi_{[-r_\varepsilon,r_\varepsilon]^n}; \\
	f_2 &= |f|\chi_{\rn\backslash[-r_\varepsilon,r_\varepsilon]^n}; \\
	f_3 &= |f\chi_{[-r_\varepsilon,r_\varepsilon]^n}-\widetilde f_1|; \\
	f_4 &= \frac\varepsilon{(2r_\varepsilon)^n}\chi_{[-r_\varepsilon,r_\varepsilon]^n}.
\end{align*}
Therefore
$$f_1-f_3-f_4 \le |f| \le f_1+f_2+f_3$$
and
$$\|f_i\|_{L^1(\rn)} \le \|f_i\|_{L_{\Phi_n}(\rn)} \le \varepsilon,\quad i=2,3,4.$$
These two facts immediately indicate that
\begin{equation}\label{eq221}
	\m_n(f_1)(x) - \sum_{i=3}^4\m_n(f_i)(x) \le \m_n(f)(x) \le \m_n(f_1)(x) + \sum_{i=2}^3\m_n(f_i)(x)
\end{equation}
and
\begin{align*}
	\|f_1\|_{L^1(\rn)} - 2\varepsilon \le \|f\|_{L^1(\rn)} \le \|f_1\|_{L^1(\rn)} + 2\varepsilon.
\end{align*}
To control the weak norm of $\m_n$, we need to introduce some notions. Let
\begin{align*}
	E_\lambda &= \lb x\in\rn:\m_n(f)(x)>\lambda\rb; \\
	E_\lambda^i &= \lb x\in\rn:\m_n(f_i)(x)>\lambda\rb,\qquad i=1,2,3,4. 
\end{align*}
Thus it follows from (\ref{eq221}) that
\begin{equation}\label{eq222}
	E_{(1+2\sqrt\varepsilon)\lambda}^1\backslash(E_{\sqrt\varepsilon\lambda}^3\cup E_{\sqrt\varepsilon\lambda}^4) \subset E_\lambda \subset E_{(1-2\sqrt\varepsilon)\lambda}^1\cup E_{\sqrt\varepsilon\lambda}^2\cup E_{\sqrt\varepsilon\lambda}^3.
\end{equation}

To prove Theorem \ref{thm1} (i), we need to consider the contribution of each term on both sides of (\ref{eq222}). Here is the main structure of this proof. The upper estimates for $E_{\sqrt\varepsilon\lambda}^2$, $E_{\sqrt\varepsilon\lambda}^3$ and $E_{\sqrt\varepsilon\lambda}^4$ will be given in Step 1. In Step 2, we are going to estabilish the lower estimate of $E_{(1+2\sqrt\varepsilon)\lambda}^1$. Combining with the upper estimates in Step 1, we may deduce the lower estimate of $E_\lambda$. In Step 3, an upper estimate for $E_{(1-2\sqrt\varepsilon)\lambda}^1$ will be given. Then the results in Step 1 and Step 3 yield an upper estimate for $E_\lambda$.

\vspace{0.3cm}
\noindent{\bf Step 1: Upper estimates for $E_{\sqrt\varepsilon\lambda}^2,E_{\sqrt\varepsilon\lambda}^3,E_{\sqrt\varepsilon\lambda}^4$.}

By the fact that $\|f_i\|_{L_{\Phi_n}(\rn)}\le\varepsilon$ for $i=2,3,4$, together with (\ref{eq21}), we obtain the upper estimates as follows:
\begin{equation}\label{eq22s1}
	|E_{\sqrt\varepsilon\lambda}^i| \le C_n\frac{1+(\log^+\frac1{\sqrt\varepsilon\lambda})^{n-1}}{\sqrt\varepsilon\lambda}\varepsilon = C_n\frac{1+(\log^+\frac1{\sqrt\varepsilon\lambda})^{n-1}}{\lambda}\sqrt\varepsilon,\qquad i=2,3,4.
\end{equation}

\vspace{0.3cm}
\noindent{\bf Step 2: Lower estimate for $E_{(1+2\sqrt\varepsilon)\lambda}^1$.}

Recalling that $f_1$ is a continuous function on $[-r_\varepsilon,r_\varepsilon]^n$, then for all $y\in[-r_\varepsilon,r_\varepsilon]^n$, we have
$$\frac\varepsilon{(2r_\varepsilon)^n} \le f_1(y) \le \max\limits_{y\in[-r_\varepsilon,r_\varepsilon]^n}f_1(y) =: A_\varepsilon <\infty.$$

Let $R_\varepsilon=(2r_\varepsilon)^{n+1}A_\varepsilon/\varepsilon+r_\varepsilon$ and define
$$E' = \lb(x_1,\cdots,x_n):|x_1|,\cdots,|x_n|>R_\varepsilon\rb.$$
From geometric view, $E'$ can be divided into $2^n$ intervals, so we denote
$$E_1' = \lb(x_1,\cdots,x_n):x_1,\cdots,x_n>R_\varepsilon\rb,$$
and the others by $E'_2,\cdots,E'_{2^n}$.

For all $x=(x_1,\cdots,x_n)\in E'_1$ and $\vec a=(a_1,\cdots,a_n),\vec b=(b_1,\cdots,b_n)$ satisfy
$$a_k\le x_k\le b_k\quad\text{and}\quad a_k<b_k,\qquad k=1,\cdots,n,$$
we define
$$F(\vec a,\vec b,x) = \frac1{\prod\limits_{k=1}^n(b_k-a_k)}\int_{a_1}^{b_1}\cdots\int_{a_n}^{b_n}f_1(y)dy.$$
Then we have the following claim.

\vspace{0.2cm}

\noindent{\bf Claim 1:} $F(\vec a,\vec b,x)$ obtains its maximum at $\vec a=(-r_\varepsilon,\cdots,-r_\varepsilon)$ and $\vec b=x$.

\vspace{0.2cm}

Note that $\text{supp}f_1 = [-r_\varepsilon,r_\varepsilon]^n$. Obviously if there exists an $a_j\ge r_\varepsilon$, then $\text{supp}f_1 \cap ([a_1,b_1]\times\cdots\times[a_n,b_n])$ is a set of measure $0$, which means $F(\vec a,\vec b,x)=0$. So we only have to discuss the case all $a_j<r_\varepsilon$. It's also easy to observe that $F(\vec a,\vec b,x)$ is a decreasing function of $b_j$. Since each $b_j\ge x_j>r_\varepsilon$, thus
$$F(\vec a,\vec b,x) \le F(\vec a,x,x) = \frac1{\prod\limits_{k=1}^n(x_k-a_k)}\int_{a_1}^{r_\varepsilon}\cdots\int_{a_n}^{r_\varepsilon}f_1(y)dy.$$

For $a_j<r_\varepsilon$, one may find
\begin{itemize}
	\item If $a_j<-r_\varepsilon$, then
		$$\frac{\partial F}{\partial a_j}(\vec a,x,x) = \frac1{x_j-a_j}F(\vec a,x,x)>0;$$
	\item If $-r_\varepsilon<a_j<r_\varepsilon$, then
		\begin{align*}
			& \frac{\partial F}{\partial a_j}(\vec a,x,x) \\
			&= \frac1{\prod\limits_{k=1}^n(b_k-a_k)}\Bigg(\frac1{b_j-a_j}\int_{a_1}^{r_\varepsilon}\cdots\int_{a_n}^{r_\varepsilon}f_1(y)dy \\
			&\quad -\int\limits_{a_1}^{r_\varepsilon}\cdots\int\limits_{a_{j-1}}^{r_\varepsilon}\int\limits_{a_{j+1}}^{r_\varepsilon}\cdots\int\limits_{a_n}^{r_\varepsilon}f_1(y_1,\cdots,y_{j-1},a_j,y_{j+1},\cdots,y_n)dy_n\cdots dy_{j+1}dy_{j-1}\cdots dy_1\Bigg) \\
			&\le \frac{\prod\limits_{k\neq j}^n(r_\varepsilon-a_k)}{\prod\limits_{k=1}^n(b_k-a_k)}\left(\frac{r_\varepsilon-a_j}{x_j-a_j}A_\varepsilon-\frac{\varepsilon}{(2r_\varepsilon)^n}\right) \le \frac{\prod\limits_{k\neq j}^n(r_\varepsilon-a_k)}{\prod\limits_{k=1}^n(b_k-a_k)}\left(\frac{2r_\varepsilon}{R_\varepsilon-r_\varepsilon}A_\varepsilon-\frac{\varepsilon}{(2r_\varepsilon)^n}\right) < 0.
		\end{align*}
\end{itemize}
These arguments deduce that $F(\vec a,x,x) \le F((-r_\varepsilon,\cdots,-r_\varepsilon),x,x)$. Therefore Claim 1 is proved. 

\vspace{0.2cm}

For $x\in E'_1$, it follows from Claim 1 that
$$\m_n(f_1)(x) = \sup_{\vec a,\vec b}F(\vec a,\vec b,x) = \frac1{\prod\limits_{k=1}^n(x_k+r_\varepsilon)}\int_{[-r_\varepsilon,r_\varepsilon]^n}f_1(y)dy.$$
For any $0 < \lambda < {\|f_1\|_{L^1(\rn)}}/{((1+2\sqrt\varepsilon)(R_\varepsilon+r_\varepsilon)^n)}$, Lemma \ref{lem21} yields that
\begin{align*}
	|E_{(1+2\sqrt\varepsilon)\lambda}^1 \cap E'_1|
	&= \left|\lb x:x_1,\cdots,x_n>R_\varepsilon,\prod_{k=1}^n(x_k+r_\varepsilon)<\frac{\|f_1\|_{L^1(\rn)}}{(1+2\sqrt\varepsilon)\lambda}\rb\right| \\
	&= \sum_{k=1}^nB_{n,k}\frac{\|f_1\|_{L^1(\rn)}}{(1+2\sqrt\varepsilon)\lambda}\left(\log\frac{\|f_1\|_{L^1(\rn)}}{(1+2\sqrt\varepsilon)\lambda}\right)^{n-k} + (-1)^n(R_\varepsilon+r_\varepsilon)^n.
\end{align*}

Repeated applications of the same technique to each $E'_i$ lead to the equation
$$|E_{(1+2\sqrt\varepsilon)\lambda}^1 \cap E'_i| = \sum_{k=1}^nB_{n,k}\frac{\|f_1\|_{L^1(\rn)}}{(1+2\sqrt\varepsilon)\lambda}\left(\log\frac{\|f_1\|_{L^1(\rn)}}{(1+2\sqrt\varepsilon)\lambda}\right)^{n-k} + (-1)^n(R_\varepsilon+r_\varepsilon)^n.$$
Combining with (\ref{eq222}) and (\ref{eq22s1}), we obtain that
\begin{align*}
	|E_\lambda|
	&\ge |E_{(1+2\sqrt\varepsilon)\lambda}^1| - |E_{\sqrt\varepsilon\lambda}^3| - |E_{\sqrt\varepsilon\lambda}^4|
	\ge \sum_{i=1}^{2^n}|E_{(1+2\sqrt\varepsilon)\lambda}^1 \cap E'_i| - |E_{\sqrt\varepsilon\lambda}^3| - |E_{\sqrt\varepsilon\lambda}^4| \\
	&\ge 2^n\sum_{k=1}^nB_{n,k}\frac{\|f_1\|_{L^1(\rn)}}{(1+2\sqrt\varepsilon)\lambda}\left(\log\frac{\|f_1\|_{L^1(\rn)}}{(1+2\sqrt\varepsilon)\lambda}\right)^{n-k} + (-2)^n(R_\varepsilon+r_\varepsilon)^n \\
	&\quad - 2C_n\frac{1+(\log^+\frac1{\sqrt\varepsilon\lambda})^{n-1}}{\lambda}\sqrt\varepsilon.
\end{align*}
Multipling $\lambda/(1+(\log^+\frac1\lambda)^{n-1})$ on both sides and let $\lambda\to0^+$, we conclude that
\begin{align*}
	\varliminf_{\lambda\to0^+}\frac\lambda{1+(\log^+\frac1\lambda)^{n-1}}|E_\lambda| &\ge \frac{2^nB_{n,1}}{1+2\sqrt\varepsilon}\|f_1\|_{L^1(\rn)} - 2C_n\sqrt\varepsilon \\
	&\ge \frac{2^n}{(n-1)!(1+2\sqrt\varepsilon)}\left(\|f\|_{L^1(\rn)}-2\varepsilon\right) - 2C_n\sqrt\varepsilon.
\end{align*}
By the arbitrariness of $\varepsilon$, we deduce that
\begin{equation}\label{eq22s2}
	\varliminf_{\lambda\to0^+}\frac\lambda{1+(\log^+\frac1\lambda)^{n-1}}|E_\lambda| \ge \frac{2^n}{(n-1)!}\|f\|_{L^1(\rn)}.
\end{equation}

\vspace{0.3cm}
\noindent{\bf Step 3: Upper estimate for $E_{(1-2\sqrt\varepsilon)\lambda}^1$.}

The argument used in Step 2 also works for $|E_{(1-2\sqrt\varepsilon)\lambda}^1 \cap E'|$, one may obtain
\begin{equation}\label{eq22s31}
	|E_{(1-2\sqrt\varepsilon)\lambda}^1 \cap E'| = 2^n\sum_{k=1}^nB_{n,k}\frac{\|f_1\|_{L^1(\rn)}}{(1-2\sqrt\varepsilon)\lambda}\left(\log\frac{\|f_1\|_{L^1(\rn)}}{(1-2\sqrt\varepsilon)\lambda}\right)^{n-k} + (-2)^n(R_\varepsilon+r_\varepsilon)^n.
\end{equation}
Now we only need to consider the contribution of $|E_{(1-2\sqrt\varepsilon)\lambda}^1 \cap (\rn\backslash E')|$.

Note that $\rn\backslash E'$ can be written as
\begin{equation}\label{eq22s32}
	\begin{aligned}
		\rn\backslash E' &= \bigcup_{i=1}^n\bigcup_{\tilde x\in\mathcal A}\lb x:|x_{l_1}|,\cdots,|x_{l_i}|\le R_\varepsilon,|x_{l_{i+1}}|,\cdots,|x_{l_n}|>R_\varepsilon\rb \\
		&=: \left(\bigcup_{i=1}^{n-1}\bigcup_{\tilde x\in\mathcal A}E''_{i,\tilde x}\right) \cup [-R_\varepsilon,R_\varepsilon]^n,
	\end{aligned}
\end{equation}
where $\tilde x=(x_{l_1},\cdots,x_{l_n})$, $\mathcal A$ is the family of all permutations of $(x_1,\cdots,x_n)$, and the cardinality of $\mathcal A$ is $n!$.

Similar as in Step 2, we may split $E''_{i,\tilde x}$ into $2^{n-i}$ sets and denote
$$E''_{i,\tilde x,1} = \lb x:|x_{l_1}|,\cdots,|x_{l_i}|\le R_\varepsilon,x_{l_{i+1}},\cdots,x_{l_n}>R_\varepsilon\rb,$$
and the others by $E''_{i,\tilde x,2},\cdots,E''_{i,\tilde x,2^{n-i}}$. See Figure \ref{figE} for $3$-dimensional case.

\begin{figure}[ht]
	\centering
	\includegraphics[width=15cm]{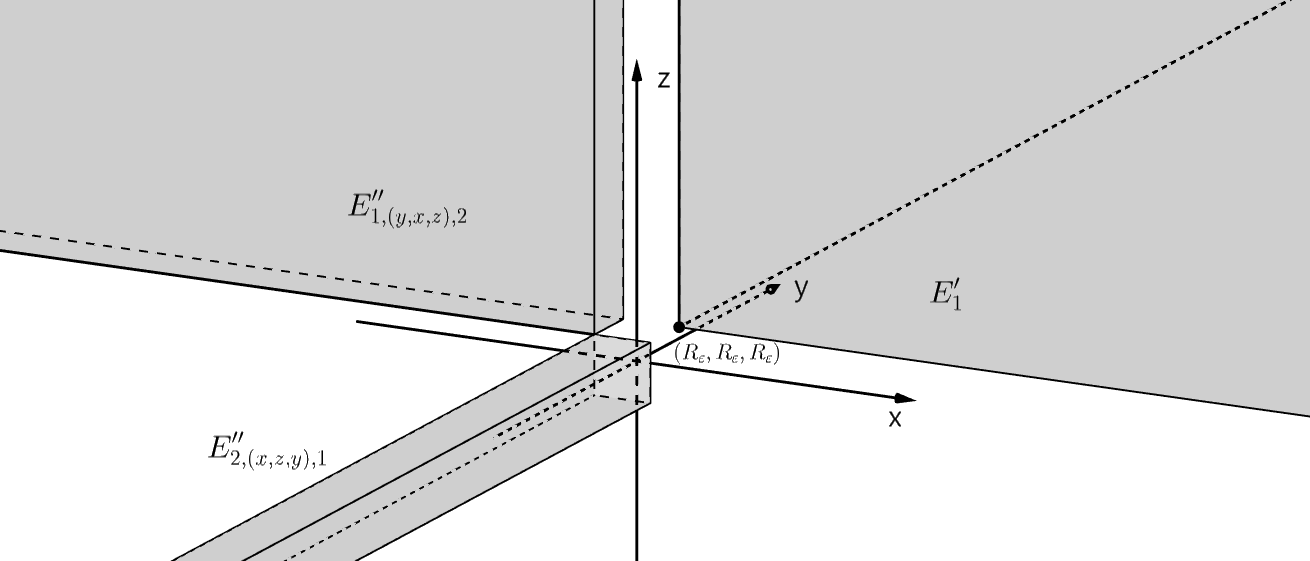}
	\caption{Part of $E'$,$E''_{1,\widetilde x}$ and $E''_{2,\widetilde x}$ in dimension 3.}
	\label{figE}
\end{figure}

Now for $x\in E''_{i,\tilde x,1}$, we define an auxiluary function $h$ which depends on $\varepsilon$ and $\widetilde x$ as
$$h(x) = A_\varepsilon\cdot\chi_{\lb x:|x_{l_1}|,\cdots,|x_{l_i}|\le R_\varepsilon,|x_{l_{i+1}}|,\cdots,|x_{l_n}|\le r_\varepsilon\rb}(x).$$
It is easy to see that $0<f_1\le h$. Then we denote
\begin{align*}
	& H(\vec a,\vec b,x) = \frac1{\prod\limits_{k=1}^n(b_{l_k}-a_{l_k})}\int_{a_{l_1}}^{b_{l_1}}\cdots\int_{a_{l_n}}^{b_{l_n}}h(y)dy \\
	&= A_\varepsilon\prod_{k=1}^i\frac{\min\lb R_\varepsilon,b_{l_k}\rb - \max\lb -R_\varepsilon,a_{l_k}\rb}{b_{l_k}-a_{l_k}}\cdot\prod_{k=i+1}^n\frac{\max\lb r_\varepsilon-a_{l_k},0\rb - \max\lb -r_\varepsilon-a_{l_k},0\rb}{b_{l_k}-a_{l_k}},
\end{align*}
and claim that:

\vspace{0.2cm}

\noindent{\bf Claim 2:} $H(\vec a,\vec b,x)$ obtains its maximum at $-R_\varepsilon\le a_{l_k}<b_{l_k}\le R_\varepsilon$ for $1\le k\le i$ and $a_{l_k}=-r_\varepsilon,b_{l_k}=x_{l_k}$ for $i+1\le k\le n$.

\vspace{0.2cm}

In fact, for $1\le k\le i$, then it is obvious that
$$\frac{\min\lb R_\varepsilon,b_{l_k}\rb - \max\lb -R_\varepsilon,a_{l_k}\rb}{b_{l_k}-a_{l_k}} \le 1,$$
and the equal sign works only if $-R_\varepsilon\le a_{l_k}<b_{l_k}\le R_\varepsilon$. On the other hand, for $i+1\le j\le n$, the following inequality holds:
$$\frac{\max\lb r_\varepsilon-a_{l_k},0\rb - \max\lb -r_\varepsilon-a_{l_k},0\rb}{b_{l_k}-a_{l_k}} \le \frac{2r_\varepsilon}{x_{l_k}+r_\varepsilon},$$
and the equality works only if $a_{l_k}=-r_\varepsilon$ and $b_{l_k}=x_{l_k}$. Then Claim 2 was proved.

\vspace{0.2cm}

By Claim 2, it follows that
$$\m_n(f_1)(x) \le \m_n(h)(x) = \sup_{\vec a,\vec b}H(\vec a,\vec b,x) = A_\varepsilon(2r_\varepsilon)^{n-i}\prod_{k=i+1}^n\frac1{x_{l_k}+r_\varepsilon}.$$
Therefore for $\lambda < A_\varepsilon(2r_\varepsilon)^{n-i}/((1-2\sqrt\varepsilon)(R_\varepsilon+r_\varepsilon)^{n-i})$, we have
\begin{align*}
	& |E_{(1-2\sqrt\varepsilon)\lambda}^1 \cap E''_{i,\tilde x,1}| \\
	&\le \left|\lb x:|x_{l_1}|,\cdots,|x_{l_i}|\le R_\varepsilon,x_{l_{i+1}},\cdots,x_{l_n}>R_\varepsilon,\prod_{k=i+1}^n{x_{l_k}+r_\varepsilon}<\frac{A_\varepsilon(2r_\varepsilon)^{n-i}}{(1-2\sqrt\varepsilon)\lambda}\rb\right| \\
	&= (2R_\varepsilon)^i\cdot\Bigg[\sum_{k=1}^{n-i}B_{n-i,k}\frac{A_\varepsilon(2r_\varepsilon)^{n-i}}{(1-2\sqrt\varepsilon)\lambda}\left(\log\frac{A_\varepsilon(2r_\varepsilon)^{n-i}}{(1-2\sqrt\varepsilon)\lambda}\right)^{n-i-k} + (-1)^{n-i}(R_\varepsilon+r_\varepsilon)^{n-i}\Bigg].
\end{align*}
Similarly, each of $|E_{(1-2\sqrt\varepsilon)\lambda}^1 \cap E''_{i,\tilde x,2}|,\cdots,|E_{(1-2\sqrt\varepsilon)\lambda}^1 \cap E''_{i,\tilde x,2^{n-i}}|$ enjoys the same bound. Therefore
\begin{align*}
	& |E_{(1-2\sqrt\varepsilon)\lambda}^1 \cap E''_{i,\tilde x}| \\
	&\le (2R_\varepsilon)^i\cdot\Bigg[\sum_{k=1}^{n-i}B_{n-i,k}\frac{A_\varepsilon(4r_\varepsilon)^{n-i}}{(1-2\sqrt\varepsilon)\lambda}\left(\log\frac{A_\varepsilon(2r_\varepsilon)^{n-i}}{(1-2\sqrt\varepsilon)\lambda}\right)^{n-i-k} + (-2)^{n-i}(R_\varepsilon+r_\varepsilon)^{n-i}\Bigg].
\end{align*}
Hence by (\ref{eq22s32}), we get
\begin{equation}\label{eq22s33}
	\begin{aligned}
		& |E_{(1-2\sqrt\varepsilon)\lambda}^1 \cap (\rn\backslash E')| \le \sum_{i=1}^{n-1}\sum_{\widetilde x\in\mathcal A}|E_{(1-2\sqrt\varepsilon)\lambda}^1 \cap E''_{i,\tilde x}| + |[-R_\varepsilon,R_\varepsilon]^n| \\
		&\le \sum_{i=1}^{n-1}n!(2R_\varepsilon)^i\cdot\Bigg[\sum_{k=1}^{n-i}B_{n-i,k}\frac{A_\varepsilon(4r_\varepsilon)^{n-i}}{(1-2\sqrt\varepsilon)\lambda}\left(\log\frac{A_\varepsilon(2r_\varepsilon)^{n-i}}{(1-2\sqrt\varepsilon)\lambda}\right)^{n-i-k}\Bigg] \\
		&\quad + \sum_{i=1}^{n-1}n!(2R_\varepsilon)^i(-2)^{n-i}(R_\varepsilon+r_\varepsilon)^{n-i} +(2R_\varepsilon)^n.
	\end{aligned}
\end{equation}

Now it follows from (\ref{eq222}), (\ref{eq22s1}), (\ref{eq22s31}) and (\ref{eq22s33}) that
\begin{align*}
	|E_\lambda| &= |E_{(1-2\sqrt\varepsilon)\lambda}^1 \cap E'| + |E_{(1-2\sqrt\varepsilon)\lambda}^1 \cap (\rn\backslash E')| + |E_{\sqrt\varepsilon\lambda}^2| + |E_{\sqrt\varepsilon\lambda}^3| \\
	&\le 2^n\sum_{k=1}^nB_{n,k}\frac{\|f_1\|_{L^1(\rn)}}{(1-2\sqrt\varepsilon)\lambda}\left(\log\frac{\|f_1\|_{L^1(\rn)}}{(1-2\sqrt\varepsilon)\lambda}\right)^{n-k} + (-2)^n(R_\varepsilon+r_\varepsilon)^n \\
	&\quad + \sum_{i=1}^{n-1}n!(2R_\varepsilon)^i\cdot\Bigg[\sum_{k=1}^{n-i}B_{n-i,k}\frac{A_\varepsilon(4r_\varepsilon)^{n-i}}{(1-2\sqrt\varepsilon)\lambda}\left(\log\frac{A_\varepsilon(2r_\varepsilon)^{n-i}}{(1-2\sqrt\varepsilon)\lambda}\right)^{n-i-k}\Bigg] \\
	&\quad + \sum_{i=1}^{n-1}n!(2R_\varepsilon)^i(-2)^{n-i}(R_\varepsilon+r_\varepsilon)^{n-i} + (2R_\varepsilon)^n + 2C_n\frac{1+(\log^+\frac1{\sqrt\varepsilon\lambda})^{n-1}}\lambda\sqrt\varepsilon .
\end{align*}
Multipling $\lambda/(1+(\log^+\frac1\lambda)^{n-1})$ on both sides and let $\lambda\to0^+$, we conclude that
\begin{align*}
	\varlimsup_{\lambda\to0^+}\frac\lambda{1+(\log^+\frac1\lambda)^{n-1}}|E_\lambda| &\le \frac{2^nB_{n,1}}{1-2\sqrt\varepsilon}\|f_1\|_{L^1(\rn)} + 2C_n\sqrt\varepsilon \\
	&\le \frac{2^n}{(n-1)!(1-2\sqrt\varepsilon)}\left(\|f\|_{L^1(\rn)}+2\varepsilon\right) + 2C_n\sqrt\varepsilon.
\end{align*}
Since $\varepsilon$ is arbitrary, it holds that
\begin{equation}\label{eq22s34}
	\varlimsup_{\lambda\to0^+}\frac\lambda{1+(\log^+\frac1\lambda)^{n-1}}|E_\lambda| \le \frac{2^n}{(n-1)!}\|f\|_{L^1(\rn)}.
\end{equation}

Finally, combining (\ref{eq22s2}) and (\ref{eq22s34}), we deduce that
$$\lim_{\lambda\to0^+}\frac\lambda{1+(\log^+\frac1\lambda)^{n-1}}|E_\lambda| = \frac{2^n}{(n-1)!}\|f\|_{L^1(\rn)}.$$
Then we finish the proof of Theorem \ref{thm1} (i).

\subsection{Proof of Theorem \ref{thm1} (ii)}
\ 
\newline
\indent Since $\m_n$ is bounded from $L^\infty$ to $L^\infty$, and apparently the best constant is $1$, then for all $\lambda>A_\varepsilon/(1-2\sqrt\varepsilon)$, it is easy to see
$$|E_{(1-2\sqrt\varepsilon)\lambda}^1| = 0.$$

Therefore for $\lambda>\max\lb A_\varepsilon/(1-2\sqrt\varepsilon),1/\sqrt\varepsilon\rb$, it follows from (\ref{eq33}), (\ref{eq34}) that 
$$|E_\lambda| \le |E_{(1-2\sqrt\varepsilon)\lambda}^1| + \sum_{i=2}^3|E_{\sqrt\varepsilon}^i| \le 2C_n\frac{1+(\log^+\frac1{\sqrt\varepsilon\lambda})^{n-1}}{\lambda}\sqrt\varepsilon \le 2C_n\frac{\sqrt\varepsilon}\lambda.$$
Multipling $\lambda/(1+(\log^+\frac1\lambda)^{n-1})$ on both sides and let $\lambda\to\infty$, we have
$$\varlimsup_{\lambda\to0^+}\frac\lambda{1+(\log^+\frac1\lambda)^{n-1}}|E_\lambda| \le 2C_n\sqrt\varepsilon.$$
By the arbitrariness of $\varepsilon$, it yields that
$$\lim_{\lambda\to\infty}\frac\lambda{1+(\log^+\frac1\lambda)^{n-1}}|E_\lambda| = 0.$$
This completes the proof of Theorem \ref{thm1} (ii).

\subsection{Proof of Corollary \ref{cor1}}
\ 
\newline
\indent Now we are ready to prove Corollary \ref{cor1}. Since the family of functions satisfying $f\in L_{\Phi_n}(\rn)$ and $\|f\|_{L_{\Phi_n}(\rn)} = \|f\|_{L^1(\rn)}$ is nonempty, therefore
\begin{align*}
	\|\m_n\|_{L_{\Phi_n}(\rn)\to L_{\Phi_n}^{1,\infty}(\rn)} &= \sup_{f\in L_{\Phi_n}(\rn)}\frac{\|\m_nf\|_{L_{\Phi_n}^{1,\infty}(\rn)}}{\|f\|_{L_{\Phi_n}(\rn)}} \\
	&\ge \sup_{\substack{f\in L_{\Phi_n}(\rn) \\ \|f\|_{L_{\Phi_n}(\rn)} = \|f\|_{L^1(\rn)}}}\frac{\|\m_nf\|_{L_{\Phi_n}^{1,\infty}(\rn)}}{\|f\|_{L^1(\rn)}} \ge \frac{2^n}{(n-1)!},
\end{align*}
where the last inequality is a direct consequence of Theorem \ref{thm1} (i).

On the other hand, note that
$$\lim_{\lambda\to1^-}\frac\lambda{1+(\log^+\frac1\lambda)^{n-1}}|\lb x\in\rn:\m_n(\chi_{B(0,1)})(x)>\lambda\rb|=\|\chi_{B(0,1)}\|_{L_{\Phi_n}(\rn)},$$
then it follows that
$$\|\m_n\|_{L_{\Phi_n}(\rn)\to L_{\Phi_n}^{1,\infty}(\rn)} \ge \max\lb\frac{2^n}{(n-1)!},1\rb.$$

It is easy to verify that
$$\lim_{\lambda\to1^-}\frac\lambda{1+(\log^+\frac1\lambda)^{n-1}}|\lb x\in\rn:\m_n^c(\chi_{B(0,1)})(x)>\lambda\rb|=\|\chi_{B(0,1)}\|_{L_{\Phi_n}(\rn)},$$
which indicates that
$$\|\m_n^c\|_{L_{\Phi_n}(\rn)\to L_{\Phi_n}^{1,\infty}(\rn)} \ge1.$$


\section{Results for multilinear strong maximal operators}\label{sec3}

As a natural generalization of linear case, the multilinear strong maximal operator have been paid lots of attentions. It was first introduced by Grafakos et al. in \cite{GLPT2011}:
$$\m_n^{(m)}(f_1,\cdots,f_m)(x) = \sup_{\substack{R \ni x \\ R\in\mathcal{R}}}\prod_{i=1}^m\frac1{|R|}\int_R|f_i(y)|dy.$$
The strong boundedness, endpoint weak type boundedness and weighted boundedness has been given. Subsequently, similar results was extented to multilinear fractional strong maximal operator by Cao et al. \cite{CXY2017,CXY2018,CXY2019}. For more works about $\m_n^{(m)}$, we refer the readers to \cite{LXY2020,ZSX2019,ZX2020}.

It is quiet natural to ask the following question:

\noindent{\bf Question}: what kinds of limiting weak type behavior does the multilinear strong maximal operator enjoy?

In this section, we are devoted to study this question. Since the difference between $m$-linear case and bilinear case is not essential, we only demonstrate the bilinear case.

\begin{thm}\label{thm3}
Let $f,g\in L_{\Phi_n}(\rn)$, then we have
\begin{equation}\label{eq31}
	\begin{aligned}
		\lim_{\lambda\to0^+}\frac\lambda{1+(\log^+\frac1\lambda)^{n-1}}| & \lbrace x\in\rn:\m_n^{(2)}(f,g)(x)>\lambda^2\rbrace| \\
		&= \frac{2^n}{(n-1)!}(\|f\|_{L^1(\rn)}\|g\|_{L^1(\rn)})^{1/2};
	\end{aligned}
\end{equation}
and
\begin{equation}\label{eq32}
	\lim_{\lambda\to\infty}\frac\lambda{1+(\log^+\frac1\lambda)^{n-1}}|\lbrace x\in\rn:\m_n^{(2)}(f,g)(x)>\lambda^2\rbrace|=0.
\end{equation}
\end{thm}

\begin{proof}
The notations in Section \ref{sec3} will continue to be used in this proof. We may still assume $\|f\|_{L^1(\rn)},\|g\|_{L^1(\rn)}>0$. There also exist functions $g_1,g_2,g_3,g_4$ for $g$ similarly as $f_1,f_2,f_3,f_4$ for $f$. We may assume $\max\limits_{y\in[-r_\varepsilon,r_\varepsilon]^n}g_1(y) \le A_\varepsilon$, otherwise we can take $A_\varepsilon = \max\limits_{y\in[-r_\varepsilon,r_\varepsilon]^n}g_1(y)$.

By the sublinearity of $\m_n^{(2)}$, it is easy to see
\begin{align*}
	\m_n^{(2)}(f,g)(x) \le \m_n^{(2)}(f_1,g_1)(x) &+ \sum_{i=2}^3\left(\m_n^{(2)}(f_1,g_i)(x) + \m_n^{(2)}(f_i,g_1)(x)\right) \\
	&+ \sum_{i=2}^3\sum_{j=2}^3\m_n^{(2)}(f_i,g_j)(x)
\end{align*}
and
\begin{align*}
	\m_n^{(2)}(f,g)(x) \ge \m_n^{(2)}(f_1,g_1)(x) &- \sum_{i=3}^4\left(\m_n^{(2)}(f,g_i)(x) + \m_n^{(2)}(f_i,g)(x)\right) \\
	&- \sum_{i=3}^4\sum_{j=3}^4\m_n^{(2)}(f_i,g_j)(x).
\end{align*}

So we define
\begin{align*}
	& \widetilde E_\lambda = \lbrace x\in\rn:\m_n^{(2)}(f,g)(x) > \lambda\rbrace; \\
	& \widetilde E_\lambda^1 = \lbrace x\in\rn:\m_n^{(2)}(f_1,g_1) > \lambda\rbrace; \\
	& \widetilde E_\lambda^2 = \bigg\lbrace x\in\rn:\sum_{i=2}^3\left(\m_n^{(2)}(f_1,g_i)(x) + \m_n^{(2)}(f_i,g_1)(x)\right) > \lambda\bigg\rbrace; \\
	& \widetilde E_\lambda^3 = \bigg\lbrace x\in\rn:\sum_{i=3}^4\left(\m_n^{(2)}(f,g_i)(x) + \m_n^{(2)}(f_i,g)(x)\right) > \lambda\bigg\rbrace; \\
	& \widetilde E_\lambda^4 = \bigg\lbrace x\in\rn:\sum_{i=2}^4\sum_{j=2}^4\m_n^{(2)}(f_i,g_j)(x) >\lambda\bigg\rbrace.
\end{align*}
Therefore the following inlcuding relationships hold:
\begin{equation}\label{eq33}
	\widetilde E_{(1+2\sqrt\varepsilon)\lambda^2}^1 \backslash (\widetilde E_{\sqrt\varepsilon\lambda^2}^3 \cup \widetilde E_{\sqrt\varepsilon\lambda^2}^4) \subset \widetilde E_{\lambda^2} \subset \widetilde E_{(1-2\sqrt\varepsilon)\lambda^2}^1 \cup \widetilde E_{\sqrt\varepsilon\lambda^2}^2 \cup \widetilde E_{\sqrt\varepsilon\lambda^2}^4
\end{equation}

We also divide this proof into four parts. The upper estimates for $\widetilde E_{\sqrt\varepsilon\lambda^2}^2$, $\widetilde E_{\sqrt\varepsilon\lambda^2}^3$ and $\widetilde E_{\sqrt\varepsilon\lambda^2}^4$ will be given in Step 1 and Step 2. Step 3 and Step 4 are devoted to demonstrate the lower and upper estimates of $\widetilde E_{(1+2\sqrt\varepsilon)\lambda^2}^1$ and $\widetilde E_{(1-2\sqrt\varepsilon)\lambda^2}^1$.

\vspace{0.3cm}
\noindent{\bf Step 1: Upper estimate for $\widetilde E_{\sqrt\varepsilon\lambda^2}^4$.}

A basic fact $\m_n^{(2)}(f_i,g_i)(x) \le \m_n(f_i)(x)\cdot\m_n(g_i)(x)$ yields that
\begin{align*}
	\widetilde E_{\sqrt\varepsilon\lambda^2}^4 &\subset \bigcup_{i=2}^4\bigcup_{j=2}^4\lb x\in\rn:\m_n^{(2)}(f_i,g_j)>\frac{\sqrt\varepsilon\lambda^2}9\rb \\
	&\subset \bigcup_{i=2}^4\bigcup_{j=2}^4\left(\lb x\in\rn:\m_n(f_i)(x)>\frac{\varepsilon^{1/4}\lambda}3\rb \cup \lb x\in\rn:\m_n(g_j)(x)>\frac{\varepsilon^{1/4}\lambda}3\rb\right) \\
	&= \bigcup_{i=2}^4\left(\lb x\in\rn:\m_n(f_i)(x)>\frac{\varepsilon^{1/4}\lambda}3\rb \cup \lb x\in\rn:\m_n(g_i)(x)>\frac{\varepsilon^{1/4}\lambda}3\rb\right).
\end{align*}
Recall that for $i=2,3,4$, $\|f_i\|_{L_{\Phi_n}(\rn)},\|g_i\|_{L_{\Phi_n}(\rn)} \le \varepsilon$. Thus it follows from (\ref{eq21}) that
\begin{equation}\label{eq34}
	\begin{aligned}
		& |\widetilde E_{\sqrt\varepsilon\lambda^2}^4| \\
		&\le \sum_{i=2}^4\left(\left|\lb x\in\rn:\m_n(f_i)(x)>\frac{\varepsilon^{1/4}\lambda}3\rb\right| + \left|\lb x\in\rn:\m_n(g_i)(x)>\frac{\varepsilon^{1/4}\lambda}3\rb\right|\right) \\
		&\le 6C_n\frac{1+(\log^+\frac3{\varepsilon^{1/4}\lambda})^{n-1}}{\varepsilon^{1/4}\lambda/3}\varepsilon \le 18C_n\frac{1+(\log^+\frac3{\varepsilon^{1/4}\lambda})^{n-1}}\lambda\varepsilon^{3/4} \\
		&\le 18C_n\frac{1+(\log^+\frac3{\varepsilon^{3/4}\lambda})^{n-1}}\lambda\varepsilon^{1/4}.
	\end{aligned}
\end{equation}
So we get the upper estimate for $\widetilde E_{\sqrt\varepsilon\lambda^2}^4$.

\vspace{0.3cm}
\noindent{\bf Step 2: Upper estimates for $\widetilde E_{\sqrt\varepsilon\lambda^2}^2$ and $\widetilde E_{\sqrt\varepsilon\lambda^2}^3$.}

Since $f_1$ is controlled by $|f|+f_3+f_4$, consequently, it holds that
\begin{align*}
	\|f_1\|_{L_{\Phi_n}(\rn)}
	&\le \int_\rn\Phi_n(|f(y)|+f_3(y)+f_4(y))dy \\
	&\le \int_{|f|=\max\lb|f|,f_3,f_4\rb}\Phi_n(3|f(y)|)dy + \sum_{i=3}^4\int_{|f_i|=\max\lb|f|,f_3,f_4\rb}\Phi_n(3|f_i(y)|)dy.
\end{align*}
The same reasoning as in the beginning of Section \ref{sec3} yields that
\begin{align*}
	\|f_1\|_{L_{\Phi_n}(\rn)} &\le 2^n3(1+(\log3)^{n-1})\left(\|f\|_{L_{\Phi_n}(\rn)} + \sum_{i=3}^4\|f_i\|_{L_{\Phi_n}(\rn)}\right) \\
	&\le 2^{2n+2}\left(\|f\|_{L_{\Phi_n}(\rn)} + 2\varepsilon\right)
	\le 2^{2n+3}\|f\|_{L_{\Phi_n}(\rn)},
\end{align*}
where the last inequality follows from $0<\varepsilon\ll\|f\|_{L^1(\rn)}$. Similarly inequality also holds for $g_1$.

It is easy to see that
\begin{align*}
	\widetilde E_{\sqrt\varepsilon\lambda^2}^2
	&\subset \bigcup_{i=2}^3\bigg(\bigg\lbrace x\in\rn:\m_n^{(2)}(f_1,g_i)(x)>\frac{\sqrt\varepsilon\lambda^2}4\bigg\rbrace \cup \bigg\lbrace x\in\rn:\m_n^{(2)}(f_i,g_1)(x)>\frac{\sqrt\varepsilon\lambda^2}4\bigg\rbrace\bigg) \\
	&\subset \bigcup_{i=2}^3\Bigg(\bigg\lbrace x\in\rn:\m_n(f_1)(x)>\frac\lambda{2\varepsilon^{1/4}}\bigg\rbrace \cup \bigg\lbrace x\in\rn:\m_n(g_i)(x)>\frac{\varepsilon^{3/4}\lambda}2\bigg\rbrace \\
	&\qquad\qquad \cup \bigg\lbrace x\in\rn:\m_n(f_i)(x)>\frac{\varepsilon^{3/4}\lambda}2\bigg\rbrace \cup \bigg\lbrace x\in\rn:\m_n(g_1)(x)>\frac\lambda{2\varepsilon^{1/4}}\bigg\rbrace\Bigg) \\
	&= \lb x\in\rn:\m_n(f_1)(x)>\frac\lambda{2\varepsilon^{1/4}}\rb \cup \lb x\in\rn:\m_n(g_1)(x)>\frac\lambda{2\varepsilon^{1/4}}\rb \\
	&\qquad \cup \bigcup_{i=2}^3\bigg(\bigg\lbrace x\in\rn:\m_n(f_i)(x)>\frac{\varepsilon^{3/4}\lambda}2\bigg\rbrace \cup \bigg\lbrace x\in\rn:\m_n(g_i)(x)>\frac{\varepsilon^{3/4}\lambda}2\bigg\rbrace\bigg).
\end{align*}
Therefore by Lemma \ref{lem21} we can get the upper estimate for $\widetilde E_{\sqrt\varepsilon\lambda^2}^2$:
\begin{equation}\label{eq35}
	\begin{aligned}
		|\widetilde E_{\sqrt\varepsilon\lambda^2}^2|
		&\le 2C_n(\|f_1\|_{L_{\Phi_n}(\rn)} + \|g_1\|_{L_{\Phi_n}(\rn)})\frac{1+(\log^+\frac{2\varepsilon^{1/4}}\lambda)^{n-1}}\lambda\varepsilon^{1/4} \\
		&\quad + 8C_n\frac{1+(\log^+\frac2{\varepsilon^{3/4}\lambda})^{n-1}}\lambda\varepsilon^{1/4} \\
		&\le (2^{2n+4}\widetilde C_n+8C_n)\frac{1+(\log^+\frac3{\varepsilon^{3/4}\lambda})^{n-1}}\lambda\varepsilon^{1/4},
	\end{aligned}
\end{equation}
where $\widetilde C_n = C_n\left(\|f\|_{L_{\Phi_n}(\rn)} + \|g\|_{L_{\Phi_n}(\rn)}\right)$.

Applying the same method, we can also have the upper estimate for $\widetilde E_{\sqrt\varepsilon\lambda}^3$:
\begin{equation}\label{eq36}
	|\widetilde E_{\sqrt\varepsilon\lambda^2}^3| \le (2\widetilde C_n+8C_n)\frac{1+(\log^+\frac3{\varepsilon^{3/4}\lambda})^{n-1}}\lambda\varepsilon^{1/4}.
\end{equation}

\vspace{0.3cm}
\noindent{\bf Step 3: Lower estimate for $\widetilde E_{(1+2\sqrt\varepsilon)\lambda^2}^1$.}

Define $G(\vec a,\vec b,x)$ by
$$G(\vec a,\vec b,x) = \frac1{\prod\limits_{k=1}^n(b_k-a_k)}\int_{a_1}^{b_1}\cdots\int_{a_n}^{b_n}g_1(y)dy.$$

Since for $x\in E_1'$, it holds that
$$\max_{\vec a,\vec b}F(\vec a,\vec b,x) = F((-r_\varepsilon,\cdots,-r_\varepsilon),x,x),$$
$$\max_{\vec a,\vec b}G(\vec a,\vec b,x) = G((-r_\varepsilon,\cdots,-r_\varepsilon),x,x),$$
thus we have
\begin{align*}
	\m_n^{(2)}(f_1,g_1)(x) &= \sup_{\vec a,\vec b}F(\vec a,\vec b,x)G(\vec a,\vec b,x) \\
	&= \frac1{\prod\limits_{k=1}^n(x_k+r_\varepsilon)^2}\int_{[-r_\varepsilon,r_\varepsilon]^n}f_1(y)dy\int_{[-r_\varepsilon,r_\varepsilon]^n}g_1(y)dy.
\end{align*}
This implies that for $\lambda$ small enough, we obtain
\begin{align*}
	& |\widetilde E_{(1+2\sqrt\varepsilon)\lambda^2}^1 \cap E_1'|
	= \left|\lb x:x_1,\cdots,x_n>R_\varepsilon,\prod_{k=1}^n(x_k+r_\varepsilon)<\frac{\|f_1\|_{L^1(\rn)}^{1/2}\|g_1\|_{L^1(\rn)}^{1/2}}{\sqrt{1+2\sqrt\varepsilon}\lambda}\rb\right| \\
	&= \sum_{k=1}^nB_{n,k}\frac{\|f_1\|_{L^1(\rn)}^{1/2}\|g_1\|_{L^1(\rn)}^{1/2}}{(n-1)!\sqrt{1+2\sqrt\varepsilon}\lambda}\left(\log\frac{\|f_1\|_{L^1(\rn)}^{1/2}\|g_1\|_{L^1(\rn)}^{1/2}}{\sqrt{1+2\sqrt\varepsilon}\lambda}\right)^{n-k}+(-1)^n(R_\varepsilon+r_\varepsilon)^n.
\end{align*}
So does $|\widetilde E_{(1+2\sqrt\varepsilon)\lambda^2}^1 \cap E_i'|$ for $i=2,\cdots,2^n$.

Combining these with (\ref{eq33}), (\ref{eq34}) and (\ref{eq36}) yields that
\begin{align*}
	|\widetilde E_{\lambda^2}|
	&\ge 2^n\sum_{k=1}^nB_{n,k}\frac{\|f_1\|_{L^1(\rn)}^{1/2}\|g_1\|_{L^1(\rn)}^{1/2}}{(n-1)!\sqrt{1+2\sqrt\varepsilon}\lambda}\left(\log\frac{\|f_1\|_{L^1(\rn)}^{1/2}\|g_1\|_{L^1(\rn)}^{1/2}}{\sqrt{1+2\sqrt\varepsilon}\lambda}\right)^{n-k} \\
	&\quad + (-1)^n(R_\varepsilon+r_\varepsilon)^n - (2\widetilde C_n + 26C_n)\frac{1+(\log^+\frac3{\varepsilon^{3/4}\lambda})^{n-1}}\lambda\varepsilon^{1/4}.
\end{align*}
Multipling $\lambda/(1+(\log^+\frac1\lambda)^{n-1})$ on both sides and let $\lambda\to0^+$ we deduce that
\begin{align*}
	& \varliminf_{\lambda\to0^+}\frac\lambda{1+(\log^+\frac1\lambda)^{n-1}}|\widetilde E_{\lambda^2}| \ge \frac{2^nB_{n,1}\|f_1\|_{L^1(\rn)}^{1/2}\|g_1\|_{L^1(\rn)}^{1/2}}{\sqrt{1+2\sqrt\varepsilon}} - (2\widetilde C_n+26C_n)\varepsilon^{1/4} \\
	&\ge \frac{2^n}{(n-1)!\sqrt{1+2\sqrt\varepsilon}}\left(\|f\|_{L^1(\rn)}-2\varepsilon\right)^{1/2}\left(\|g\|_{L^1(\rn)}-2\varepsilon\right)^{1/2} - (2\widetilde C_n+26C_n)\varepsilon^{1/4}.
\end{align*}
By the arbitrariness of $\varepsilon$, we get the lower estimate as follows:
\begin{equation}\label{eq37}
	\varliminf_{\lambda\to0^+}\frac\lambda{1+(\log^+\frac1\lambda)^{n-1}}|\widetilde E_{\lambda^2}| \ge \frac{2^n}{(n-1)!}\|f\|_{L^1(\rn)}^{1/2}\|g\|_{L^1(\rn)}^{1/2}.
\end{equation}

\vspace{0.3cm}
\noindent{\bf Step 4: Upper estimate for $\widetilde E_{(1-2\sqrt\varepsilon)\lambda^2}^1$.}

The same arguments as in Step 3 of Section \ref{sec3} imply that
\begin{equation}\label{eq38}
	\begin{aligned}
		& |\widetilde E_{(1-2\sqrt\varepsilon)\lambda^2}^1 \cap E'| \\
		&= 2^n\sum_{k=1}^nB_{n,k}\frac{\|f_1\|_{L^1(\rn)}^{1/2}\|g_1\|_{L^1(\rn)}^{1/2}}{(n-1)!\sqrt{1-2\sqrt\varepsilon}\lambda}\left(\log\frac{\|f_1\|_{L^1(\rn)}^{1/2}\|g_1\|_{L^1(\rn)}^{1/2}}{\sqrt{1-2\sqrt\varepsilon}\lambda}\right)^{n-1} \\
		&\quad + (-1)^n(R_\varepsilon+r_\varepsilon)^n.
	\end{aligned}
\end{equation}

It is easy to verify that
\begin{align*}
	& |\widetilde E_{(1-2\sqrt\varepsilon)\lambda^2}^1 \cap (\rn\backslash E')| \le |E_{(1-2\sqrt\varepsilon)\lambda}^1 \cap (\rn\backslash E')| + |\lb x\in\rn:\m_n(g_1)(x)>\lambda\rb \cap (\rn\backslash E')| \\
	&\le |E_{(1-2\sqrt\varepsilon)\lambda}^1 \cap (\rn\backslash E')| + |\lb x\in\rn:\m_n(g_1)(x)>(1-2\sqrt\varepsilon)\lambda\rb \cap (\rn\backslash E')|.
\end{align*}
From Step 3 of Section \ref{sec2} we know that the right side multiply $\lambda/(1+(\log^+\frac1\lambda)^{n-1})$ will converge to $0$ as $\lambda\to0^+$. Then, by (\ref{eq33}), (\ref{eq34}), (\ref{eq35}) and (\ref{eq38}), we have
\begin{align*}
	& \varlimsup_{\lambda\to0^+}\frac\lambda{1+(\log^+\frac1\lambda)^{n-1}}|\widetilde E_{\lambda^2}| \le \frac{2^nB_{n,1}\|f_1\|_{L^1(\rn)}^{1/2}\|g_1\|_{L^1(\rn)}^{1/2}}{\sqrt{1+2\sqrt\varepsilon}} + (2^{2n+4}\widetilde C_n+26C_n)\varepsilon^{1/4} \\
	&\le \frac{2^n}{(n-1)!\sqrt{1+2\sqrt\varepsilon}}\left(\|f\|_{L^1(\rn)}+2\varepsilon\right)^{1/2}\left(\|g\|_{L^1(\rn)}+2\varepsilon\right)^{1/2} + (2^{2n+4}\widetilde C_n+26C_n)\varepsilon^{1/4}.
\end{align*}
The lower estimate follows from the arbitrariness of $\varepsilon$:
\begin{equation}\label{eq39}
	\varlimsup_{\lambda\to0^+}\frac\lambda{1+(\log^+\frac1\lambda)^{n-1}}|\widetilde E_{\lambda^2}| \le \frac{2^n}{(n-1)!}\|f\|_{L^1(\rn)}^{1/2}\|g\|_{L^1(\rn)}^{1/2}.
\end{equation}

Combining (\ref{eq37}) and (\ref{eq39}), we deduce that
$$\lim_{\lambda\to0^+}\frac\lambda{1+(\log^+\frac1\lambda)^{n-1}}|\widetilde E_{\lambda^2}| = \frac{2^n}{(n-1)!}\|f\|_{L^1(\rn)}^{1/2}\|g\|_{L^1(\rn)}^{1/2}.$$
The proof of (\ref{eq31}) is finished.

\vspace{0.3cm}

Finally, by Theorem \ref{thm1} (ii), (\ref{eq32}) follows from
\begin{align*}
	0 &\le \lim_{\lambda\to\infty}\frac\lambda{1+(\log^+\frac1\lambda)^{n-1}}\left|\lb x:\m_n^{(2)}(f,g)(x)>\lambda^2\rb\right| \\
	&\le \lim_{\lambda\to\infty}\frac\lambda{1+(\log^+\frac1\lambda)^{n-1}}\left(\left|\lb x:\m_n(f)(x)>\lambda\rb\right| + \left|\lb x:\m_ng(x)>\lambda\rb\right|\right) =0.
\end{align*}

\end{proof}

\end{document}